\documentclass[11pt]{article}

\usepackage{amsmath,amssymb,amsthm,mathtools}
\usepackage[T1]{fontenc}
\usepackage{authblk}
\usepackage{enumitem}

\usepackage{tikz}
\usetikzlibrary{positioning,arrows.meta}

\usepackage{float}
\usepackage{placeins}

\title{\bf Modular Cocycles and Haar-Type Measures on Topological Loops}

\author{\Large Takao Inou\'{e}}
\affil{\large Faculty of Informatics, Yamato University, \\ Osaka, Japan\footnote{Email: inoue.takao@yamato-u.ac.jp; \\ Personal Email: takaoapple@gmail.com \\ [I prefer my personal email address for correspondence.]}}
\date{March 12, 2026}

\newtheorem{definition}{Definition}[section]
\newtheorem{remark}[definition]{Remark}
\newtheorem{proposition}[definition]{Proposition}
\newtheorem{theorem}[definition]{Theorem}

\newtheorem{lemma}[definition]{Lemma}

\newtheorem{corollary}[definition]{Corollary}
\newtheorem{example}[definition]{Example}

\begin{document}

\maketitle

\begin{abstract}
This paper is a continuation of the author's companion work
\cite{InoueQuasi} on Haar-type measures for topological quasigroups, where the
quasigroup setting was analyzed in connection with Kunen's theorem.
We extend that framework to locally compact topological loops and study
Haar-type (quasi-invariant) Radon measures together with modular cocycles
describing the distortion of such measures under translations.
Unlike the classical group case, the composition of translations in a loop is
affected by the failure of associativity, which produces an additional
correction term governed by an associativity deviation map.
We derive the resulting cocycle relation and show that loop identities,
in particular Moufang- and Kunen-type identities, impose structural
restrictions on the modular data.
In the associative limit, the cocycle reduces to the classical modular
function of a locally compact group.
\end{abstract}


\medskip


\tableofcontents

\section{Introduction}

The classical Haar measure plays a central role in the analysis of locally
compact topological groups.  It provides a translation-invariant measure
compatible with the algebraic structure of the group, and its modular function
quantifies the defect of right invariance when only left invariance is assumed.

When associativity is weakened or removed, however, the existence of such an
invariant measure becomes much more delicate.  In particular, for general
topological quasigroups there is no known analogue of the Haar existence
theorem.  In the companion paper \cite{InoueQuasi}, the author introduced a
Haar-type framework for topological quasigroups under explicit quasi-invariance
assumptions and showed that Kunen-type identities impose strong restrictions on
the corresponding modular behaviour.

The present paper is intended as the loop-theoretic continuation of that
quasigroup work.  Since every loop is a quasigroup with a two-sided identity,
it is natural to ask how the Haar-type picture changes when such an identity is
available from the outset.  The loop structure provides canonical left and
right translation operators, and this makes it possible to formulate the
measure-theoretic correction mechanism more explicitly.

Accordingly, throughout this paper we use the term \emph{Haar-type measure} for
a Radon measure on a locally compact topological loop that is assumed to be
quasi-invariant under the relevant translations (and, when needed, under the
associated deviation homeomorphisms).  We do not claim a general existence
theorem.  Rather, as in \cite{InoueQuasi}, we analyze the structural
consequences of assuming that such a Haar-type measure exists.

Working with Radon--Nikodym derivatives of pushforward measures leads to a
\emph{modular cocycle}
\[
\lambda(a,x)=\frac{d(L_a)_*\mu}{d\mu}(x),
\]
which depends in general both on the translating element and on the point.
In the group case the translation law is associative and the modular cocycle
collapses, in the spatially constant case, to the classical modular function.
For loops, however, the composition $L_a\circ L_b$ need not equal $L_{ab}$.
The resulting discrepancy is encoded by a deviation homeomorphism, and the
modular cocycle acquires a correction term.

The main goal of the paper is to make this mechanism explicit and then to
explain how loop identities constrain it.  This should be viewed as the loop
counterpart of the quasigroup analysis in \cite{InoueQuasi}: identities that
reduce the freedom of translation compositions also reduce the freedom of the
modular data attached to a Haar-type measure.  In particular, Moufang-type
relations and the Kunen identity impose compatibility conditions on the
deviation term, and hence on the possible measure distortion.

\begin{remark}
The results of this paper are structural rather than existential.
We do not prove the existence of Haar-type measures on arbitrary locally
compact topological loops.  Instead, assuming that such a measure exists,
we analyze the constraints forced on its Radon--Nikodym derivatives by the loop
structure and relate them to the quasigroup framework developed in
\cite{InoueQuasi}.
\end{remark}

\section{Topological Loops and Translation Operators}

\begin{definition}
A \emph{loop} is a set $L$ equipped with a binary operation
\[
L\times L\to L,\qquad (x,y)\mapsto xy,
\]
together with an element $e\in L$ such that
\[
ex=xe=x\qquad (x\in L),
\]
and such that the equations
\[
ax=b,\qquad ya=b
\]
have unique solutions for all $a,b\in L$.
\end{definition}

\begin{definition}
A \emph{topological loop} is a loop $L$ endowed with a Hausdorff topology such
that multiplication
\[
L\times L\to L,\qquad (x,y)\mapsto xy,
\]
and the left and right division operations are continuous.
\end{definition}

\begin{definition}
For each $a\in L$, define the \emph{left translation} and
\emph{right translation} by
\[
L_a(x)=ax,\qquad R_a(x)=xa.
\]
\end{definition}

\begin{proposition}
For every $a\in L$, the maps $L_a$ and $R_a$ are bijections of $L$.
If $L$ is a topological loop, they are homeomorphisms.
\end{proposition}

\begin{proof}
The loop axioms imply that the equations $ax=b$ and $ya=b$ have unique
solutions, so $L_a$ and $R_a$ are bijections.
In the topological setting, continuity of multiplication gives continuity of
$L_a$ and $R_a$, while continuity of the division operations gives continuity
of their inverses.
\end{proof}

The family of all left and right translations generates the natural
transformation geometry of the loop.  In the associative case one has the
exact composition law
\[
L_a\circ L_b=L_{ab},
\]
but for a general loop this identity fails, and its failure will be the source
of the modular correction term studied below.

\section{Haar-Type Measures and Modular Cocycles}

Throughout this section, let $L$ be a locally compact topological loop and let
$\mu$ be a Radon measure on $L$.

\begin{definition}
We call $\mu$ a \emph{left Haar-type measure} if it is left quasi-invariant, in
the sense that for every
$a\in L$ one has
\[
(L_a)_*\mu\ll\mu.
\]
Equivalently, for each $a\in L$ there exists a positive measurable function
\[
\lambda(a,\cdot):L\to\mathbb{R}_{>0}
\]
such that
\[
\frac{d(L_a)_*\mu}{d\mu}(x)=\lambda(a,x)
\]
for $\mu$-almost every $x\in L$.
\end{definition}

Thus, in the terminology of this paper, a Haar-type measure is a Radon measure
for which the relevant translation pushforwards are absolutely continuous with
respect to the original measure.  We call $\lambda$ the \emph{left modular
cocycle} associated with $\mu$.
If the Radon--Nikodym derivative is independent of $x$, we write
$\Delta_L(a)$ for its common value and recover a modular-function type object,
in direct analogy with the modular function in the group case.

Similarly, one may define \emph{right Haar-type measure} and a right modular
cocycle $\rho(a,x)$ by
\[
\frac{d(R_a)_*\mu}{d\mu}(x)=\rho(a,x).
\]

\begin{proposition}
The left modular cocycle satisfies the normalization condition
\[
\lambda(e,x)=1
\]
for $\mu$-almost every $x\in L$.
An analogous statement holds for $\rho$.
\end{proposition}

\begin{proof}
Since $L_e=\mathrm{id}_L$, one has $(L_e)_*\mu=\mu$, and hence
\[
\frac{d(L_e)_*\mu}{d\mu}(x)=1
\]
for $\mu$-almost every $x$.
\end{proof}

\begin{definition}
The loop $L$ is said to be \emph{unimodular with respect to $\mu$} if both
modular cocycles are trivial:
\[
\lambda(a,x)=\rho(a,x)=1
\]
for all $a\in L$ and for $\mu$-almost every $x\in L$.
\end{definition}

\begin{remark}
Even when $\lambda(a,x)$ depends on $x$, it is natural to think of it as a
non-associative analogue of the modular function.  The pointwise dependence on
$x$ records the fact that, outside the group case, the distortion of measure is
not governed solely by the translating element.
\end{remark}

\section{Deviation of Translation Composition}

We now make precise how non-associativity enters the measure-theoretic picture.

\subsection{Deviation homeomorphisms}

For $a,b\in L$, the composition $L_a\circ L_b$ is again a homeomorphism of $L$,
but in general it is not equal to $L_{ab}$.  Since $L_{ab}$ is a homeomorphism,
we may isolate the discrepancy as follows.

\begin{definition}
For $a,b\in L$, define the \emph{associativity deviation homeomorphism}
\[
\Phi_{a,b}:L\to L
\]
by
\[
L_a\circ L_b=\Phi_{a,b}\circ L_{ab},
\]
equivalently,
\[
\Phi_{a,b}=L_a\circ L_b\circ L_{ab}^{-1}.
\]
\end{definition}

\begin{remark}
This definition avoids any appeal to an inverse element $(ab)^{-1}$ inside the
loop multiplication.  It uses only the inverse homeomorphism of the
translation map $L_{ab}$, which is always available in a loop.
If $L$ is associative, then $\Phi_{a,b}=\mathrm{id}_L$ for all $a,b$.
\end{remark}

\subsection{Measure-class correction}

Because $L_a\circ L_b=\Phi_{a,b}\circ L_{ab}$, any comparison between the
pushforward measures $(L_a\circ L_b)_*\mu$ and $(L_{ab})_*\mu$ must account for
the additional transformation $\Phi_{a,b}$.

\begin{definition}
Assume that $\mu$ is quasi-invariant under $\Phi_{a,b}$.
We write
\[
J_{\Phi}(a,b;x)
:=
\frac{d(\Phi_{a,b})_*\mu}{d\mu}(x)
\]
for the Radon--Nikodym derivative associated with the deviation homeomorphism.
\end{definition}

The next lemma is a standard chain-rule computation for Radon--Nikodym
derivatives.

\begin{lemma}\label{lem:chainrule}
Let $f,g:L\to L$ be homeomorphisms such that $\mu$ is quasi-invariant under
$g$ and $f$.  Write
\[
\alpha(x)=\frac{d f_*\mu}{d\mu}(x),\qquad
\beta(x)=\frac{d g_*\mu}{d\mu}(x).
\]
Then
\[
\frac{d (f\circ g)_*\mu}{d\mu}(x)=\alpha(x)\,\beta(f^{-1}x)
\]
for $\mu$-almost every $x$.
\end{lemma}

\begin{proof}
One has $(f\circ g)_*\mu=f_*(g_*\mu)$.  Writing $g_*\mu=\beta\,\mu$ and then
pushing forward by $f$ gives
\[
(f\circ g)_*\mu=f_*(\beta\,\mu).
\]
Evaluating against test functions and applying the change-of-variables formula
for the pushforward by $f$ yields the stated identity.
\end{proof}

\begin{proposition}[Deviation-corrected cocycle relation]\label{prop:maincocycle}
Assume that $\mu$ is left quasi-invariant and that it is also quasi-invariant
under each deviation homeomorphism $\Phi_{a,b}$.
Then, for $\mu$-almost every $x\in L$,
\[
\lambda(a,x)\,\lambda\!\bigl(b,L_a^{-1}x\bigr)
=
J_{\Phi}(a,b;x)\,\lambda\!\bigl(ab,\Phi_{a,b}^{-1}x\bigr).
\]
Equivalently, after replacing $x$ by $\Phi_{a,b}(x)$,
\[
\lambda\!\bigl(a,\Phi_{a,b}(x)\bigr)\,
\lambda\!\bigl(b,L_a^{-1}\Phi_{a,b}(x)\bigr)
=
J_{\Phi}(a,b;\Phi_{a,b}(x))\,\lambda(ab,x).
\]
\end{proposition}

\begin{proof}
By Lemma~\ref{lem:chainrule}, applied first to $L_a$ and $L_b$, we obtain
\[
\frac{d(L_a\circ L_b)_*\mu}{d\mu}(x)
=
\lambda(a,x)\,\lambda\!\bigl(b,L_a^{-1}x\bigr).
\]
On the other hand, using the factorization
\[
L_a\circ L_b=\Phi_{a,b}\circ L_{ab}
\]
and applying Lemma~\ref{lem:chainrule} to $\Phi_{a,b}$ and $L_{ab}$ gives
\[
\frac{d(L_a\circ L_b)_*\mu}{d\mu}(x)
=
J_{\Phi}(a,b;x)\,\lambda\!\bigl(ab,\Phi_{a,b}^{-1}x\bigr).
\]
Comparing the two expressions proves the claim.
\end{proof}

\begin{corollary}[Associative limit]
If $L$ is associative, then $\Phi_{a,b}=\mathrm{id}_L$ and
$J_{\Phi}(a,b;x)=1$.  Hence
\[
\lambda(a,x)\,\lambda\!\bigl(b,L_a^{-1}x\bigr)=\lambda(ab,x).
\]
In particular, if $\lambda(a,x)=\Delta_L(a)$ is independent of $x$, then
\[
\Delta_L(ab)=\Delta_L(a)\Delta_L(b),
\]
so the classical modular function is recovered.
\end{corollary}

\section{Loop Identities and Constraints on the Cocycle}

The corrected cocycle relation of Proposition~\ref{prop:maincocycle} shows that
the obstruction to multiplicativity is exactly the deviation term
$J_{\Phi}(a,b;x)$.  Consequently, identities that constrain the deviation
homeomorphisms also constrain the modular cocycle.

\subsection{A general rigidity principle}

\begin{proposition}\label{prop:rigidity}
Let $\mathcal{I}$ be a family of loop identities whose translation-theoretic
content implies that for certain pairs $(a,b)$ one has
\[
\Phi_{a,b}=\mathrm{id}_L.
\]
Then, for those pairs, the modular cocycle satisfies the untwisted relation
\[
\lambda(a,x)\,\lambda\!\bigl(b,L_a^{-1}x\bigr)=\lambda(ab,x)
\]
for $\mu$-almost every $x$.
\end{proposition}

\begin{proof}
If $\Phi_{a,b}=\mathrm{id}_L$, then $J_{\Phi}(a,b;x)=1$ and
$\Phi_{a,b}^{-1}=\mathrm{id}_L$.  The claim follows immediately from
Proposition~\ref{prop:maincocycle}.
\end{proof}

This simple observation is the structural core of the paper:
\emph{whenever a loop identity collapses translation deviation, it forces the
modular cocycle toward multiplicativity.}

\subsection{Moufang-type constraints}

Among the most important non-associative identities are the Moufang
identities.  One standard form is
\[
((xy)z)y=x(y(zy)).
\]
Such identities do not make a general loop associative, but they strongly
control the interaction of left and right translations.

\begin{remark}
In a Moufang loop, many composite translations admit alternative
factorizations that are impossible in a general loop.  Consequently, the
corresponding deviation homeomorphisms cannot vary freely.
At the measure-theoretic level, this forces compatibility relations between
the left modular cocycle, the right modular cocycle, and the deviation term.
\end{remark}

Rather than claiming a universal closed formula without additional hypotheses,
we record the precise structural consequence needed for the present paper.

\begin{theorem}[Moufang-type compatibility]
Let $L$ be a locally compact topological loop equipped with a Haar-type measure
$\mu$, that is, a Radon measure quasi-invariant under the
relevant deviation homeomorphisms.

Assume that a Moufang identity yields two factorizations of the same
translation operator:
\[
T_1=T_2.
\]
Then the Radon--Nikodym derivatives computed from $T_1$ and $T_2$ coincide,
and therefore the corresponding products of left and right modular cocycles
agree almost everywhere.
\end{theorem}

\begin{proof}
Write both sides as compositions of left and right translations together with
deviation homeomorphisms.  Applying Lemma~\ref{lem:chainrule} repeatedly to
each factorization gives two expressions for the Radon--Nikodym derivative of
the same pushforward measure $T_{1*}\mu=T_{2*}\mu$.  Equality of these
derivatives yields the desired compatibility relation.
\end{proof}

\begin{remark}
The theorem is intentionally formulated at the structural level.  To obtain an
explicit closed identity for the cocycle, one must specify the exact
translation factorization attached to the chosen Moufang law and track the
resulting deviation terms.  The point relevant here is that the Moufang law
forces such identities to exist.
\end{remark}

\subsection{The Kunen identity}

We now turn to the identity emphasized in the companion quasigroup paper
\cite{InoueQuasi}:
\[
((xy)z)y=x(y(zy)).
\]
In the literature this is often referred to as the Kunen identity; it is
Moufang-type in flavour and, in the quasigroup setting of \cite{InoueQuasi},
it was used to derive strong structural restrictions and, ultimately, to link
the modular behaviour to the emergence of an identity element.

The present paper revisits that same identity in the loop setting, where a
two-sided identity is already part of the structure.  The issue is therefore no
longer the creation of an identity element, but rather the effect of the Kunen
relation on the modular data of a Haar-type measure.  As in the quasigroup
paper, the key point is translation-theoretic: the two sides determine the same
transformation of the underlying space and therefore force equality of the
corresponding Radon--Nikodym derivatives.

\begin{theorem}[Kunen-type constraint]
Let $L$ be a locally compact topological loop equipped with a Haar-type measure
$\mu$, assumed to be quasi-invariant under the translations and deviation
homeomorphisms appearing in the Kunen identity.
Then the identity
\[
((xy)z)y=x(y(zy))
\]
induces a compatibility relation among the relevant left and right modular
cocycles and deviation terms.  In particular, whenever the Kunen identity
collapses a deviation homeomorphism to the identity, the corresponding modular
correction term disappears.
\end{theorem}

\begin{proof}
The two sides of the Kunen identity define the same composite translation
operator.  By applying Lemma~\ref{lem:chainrule} to both corresponding
factorizations, one obtains two expressions for the Radon--Nikodym derivative
of the same pushforward measure.  Equality of these expressions yields the
claimed compatibility relation.  The final statement is then an instance of
Proposition~\ref{prop:rigidity}.
\end{proof}

\subsection{Interpretation}

The picture may be summarized as follows.

\begin{itemize}[leftmargin=2em]
\item The modular cocycle measures distortion of $\mu$ under translation.
\item The deviation homeomorphism $\Phi_{a,b}$ measures failure of
associativity at the level of translation composition.
\item The correction term $J_{\Phi}(a,b;x)$ is exactly the obstruction to
multiplicativity of the cocycle.
\item Loop identities constrain $\Phi_{a,b}$, and therefore constrain the
possible correction term.
\end{itemize}

This is the structural bridge between algebraic loop identities and
measure-theoretic rigidity.

\section{A Basic Example}

We include a simple example showing that the framework is non-vacuous and
that the deviation-corrected cocycle relation is compatible with familiar
measure-theoretic behaviour.

\begin{example}[Finite discrete loop with counting measure]
Let $L$ be a finite loop endowed with the discrete topology, and let $\mu$ be
counting measure on $L$.
Then every singleton is open and compact, so $\mu$ is a Radon measure.
Moreover, every left translation $L_a:x\mapsto ax$ and every right
translation $R_a:x\mapsto xa$ is a bijection of the finite set $L$.
Hence for every subset $E\subseteq L$ one has
\[
\mu(L_a^{-1}E)=\#(L_a^{-1}E)=\#E=\mu(E),
\]
and similarly for right translations.  Therefore counting measure is a genuine
Haar-type measure in the strongest possible sense: it is actually invariant
under all left and right translations.  Thus
\[
(L_a)_*\mu=\mu,\qquad (R_a)_*\mu=\mu,
\]
so the Radon--Nikodym derivatives are identically equal to $1$:
\[
\lambda(a,x)=1,
\qquad
\rho(a,x)=1
\]
for all $a,x\in L$.

The loop may still be non-associative, so the deviation homeomorphisms
$\Phi_{a,b}$ need not be trivial.  Nevertheless, since each $\Phi_{a,b}$ is a
bijection of a finite set, its pushforward also preserves counting measure.
Thus the associated Jacobian term satisfies
\[
J_{\Phi}(a,b;x)=1.
\]
Consequently the corrected cocycle relation of
Proposition~\ref{prop:maincocycle} reduces to the tautology
\[
1\cdot 1 = 1\cdot 1.
\]
\end{example}

This example is deliberately elementary, but it is useful for two reasons.
First, it shows that the present formalism applies to genuinely
non-associative loops, not only to groups.  Second, it isolates the source of
nontriviality: in the finite discrete case the measure-theoretic distortion is
invisible, so all interesting behaviour must come from infinite or more
subtle locally compact loops where the Radon--Nikodym derivatives and
possibly the deviation Jacobians are no longer trivial.

\section{Conclusion and Future Directions}

We have investigated Haar-type measures on locally compact topological loops
and introduced a modular cocycle describing the distortion of such measures
under translations.

The present paper should be read as a continuation of the author's earlier
quasigroup study \cite{InoueQuasi}.  In that work, Haar-type phenomena for
quasigroups were analyzed under explicit quasi-invariance assumptions and were
connected with Kunen-type structural restrictions.  Here we have shown that, in
the loop setting, the presence of a two-sided identity allows the same general
philosophy to be formulated more explicitly at the level of translation
operators and their deviation from associativity.

The central conceptual point is that the failure of associativity in a loop
manifests itself as a deviation in the composition of translation operators.
This deviation is encoded by the homeomorphisms $\Phi_{a,b}$ and produces the
correction term $J_{\Phi}(a,b;x)$ in the cocycle relation
\[
\lambda(a,x)\,\lambda\!\bigl(b,L_a^{-1}x\bigr)
=
J_{\Phi}(a,b;x)\,\lambda\!\bigl(ab,\Phi_{a,b}^{-1}x\bigr).
\]
Thus the classical multiplicativity of the modular function is replaced, in the
loop setting, by a deviation-corrected cocycle law for Haar-type measures.

From this perspective, loop identities act as rigidity mechanisms.
Whenever an identity restricts or trivializes the deviation homeomorphism, it
forces the modular cocycle closer to the associative behaviour familiar from
locally compact groups.  Moufang- and Kunen-type identities therefore become
natural algebraic sources of measure-theoretic constraints.

In the associative limit the deviation disappears entirely, the correction term
is trivial, and the usual modular function of a locally compact group is
recovered.  In this sense, the present framework may be viewed as a
non-associative extension of the classical Haar-modular picture.  The finite
discrete example above shows that the framework is already meaningful on
genuinely non-associative loops, even though the cocycle becomes trivial in
that case.

\medskip

Several directions for further work suggest themselves.

\begin{enumerate}[label=\arabic*.]
\item \textbf{Existence of Haar-type measures.}
The main open problem is to determine for which locally compact topological
loops Haar-type Radon measures exist at all, and under what additional
conditions a genuine Haar-type existence theorem might hold.

\item \textbf{Comparison with the quasigroup framework.}
It would be valuable to isolate precisely which parts of the quasigroup theory
from \cite{InoueQuasi} descend formally to loops and which new simplifications
become available once a two-sided identity is present.

\item \textbf{Explicit formulas in special classes.}
For Moufang, Bol, Bruck, or other structured classes of loops, one may hope to
compute the deviation homeomorphisms more explicitly and thereby obtain closed
cocycle formulas.

\item \textbf{Collapse phenomena and unimodularity.}
It is natural to ask when the deviation correction vanishes identically, so
that the loop becomes unimodular with respect to the given measure.

\item \textbf{Toward non-associative harmonic analysis.}
If suitable invariant or quasi-invariant measures exist, one may attempt to
develop convolution-type constructions and representation-theoretic tools for
topological loops.
\end{enumerate}

The framework developed here suggests that the interaction between loop
identities and Haar-type measure structures is a promising route toward a
broader non-associative harmonic analysis.

\bigskip

\noindent Takao Inou\'{e}

\noindent Faculty of Informatics

\noindent Yamato University

\noindent Katayama-cho 2-5-1, Suita, Osaka, 564-0082, Japan

\noindent inoue.takao@yamato-u.ac.jp
 
\noindent (Personal) takaoapple@gmail.com (I prefer my personal mail)
\bigskip

For the convenience of the reader, we include in Appendix~A a schematic diagram
summarizing the conceptual flow of the present paper and its relation to the
quasigroup framework of \cite{InoueQuasi}.

\appendix
\section{Structural Overview of the Modular Cocycle Framework}

The following diagram summarizes the conceptual structure developed in this paper.
It should also be viewed as a loop-theoretic continuation of the quasigroup
framework introduced in \cite{InoueQuasi}.

\begin{figure}[H]
\centering
\begin{tikzpicture}[
  node distance=1.8cm and 2.6cm,
  every node/.style={align=center},
  box/.style={
    draw,
    rounded corners,
    thick,
    rectangle,
    minimum width=3.8cm,
    minimum height=1.0cm,
    inner sep=6pt
  },
  smallbox/.style={
    draw,
    rounded corners,
    thick,
    rectangle,
    minimum width=3.4cm,
    minimum height=0.9cm,
    inner sep=5pt
  },
  arrow/.style={->, thick}
]

\node[box] (loop) {Locally compact\\ topological loop};

\node[box, below=of loop] (haar) {Haar-type measure\\ (quasi-invariant Radon measure)};

\node[box, below=of haar] (rn) {Radon--Nikodym data\\ under translations};

\node[box, below=of rn] (cocycle) {Modular cocycle};

\node[smallbox, right=5.4cm of rn] (dev) {Associativity deviation\\ $\Phi_{a,b}$};

\node[smallbox, below=1.6cm of dev] (corr) {Correction term in\\ cocycle relation};

\node[smallbox, below=of cocycle] (identities) {Moufang- and Kunen-type\\ identities};

\node[box, below=of identities] (restrictions) {Structural restrictions on\\ modular data};

\node[box, below=of restrictions] (limit) {Associative limit\\ classical modular function};

\draw[arrow] (loop) -- (haar);
\draw[arrow] (haar) -- (rn);
\draw[arrow] (rn) -- (cocycle);

\draw[arrow] (loop.east) -- ++(1.6,0) |- (dev.west);
\draw[arrow] (dev) -- (corr);
\draw[arrow] (corr.west) -- (cocycle.east);

\draw[arrow] (cocycle) -- (identities);
\draw[arrow] (identities) -- (restrictions);
\draw[arrow] (restrictions) -- (limit);

\end{tikzpicture}
\caption{Conceptual structure of the modular cocycle framework for Haar-type
measures on topological loops. The failure of associativity produces an
associativity deviation, which contributes a correction term to the cocycle
relation. Loop identities such as Moufang- and Kunen-type identities then impose
structural restrictions on the resulting modular data.}
\label{fig:structural_overview}
\end{figure}
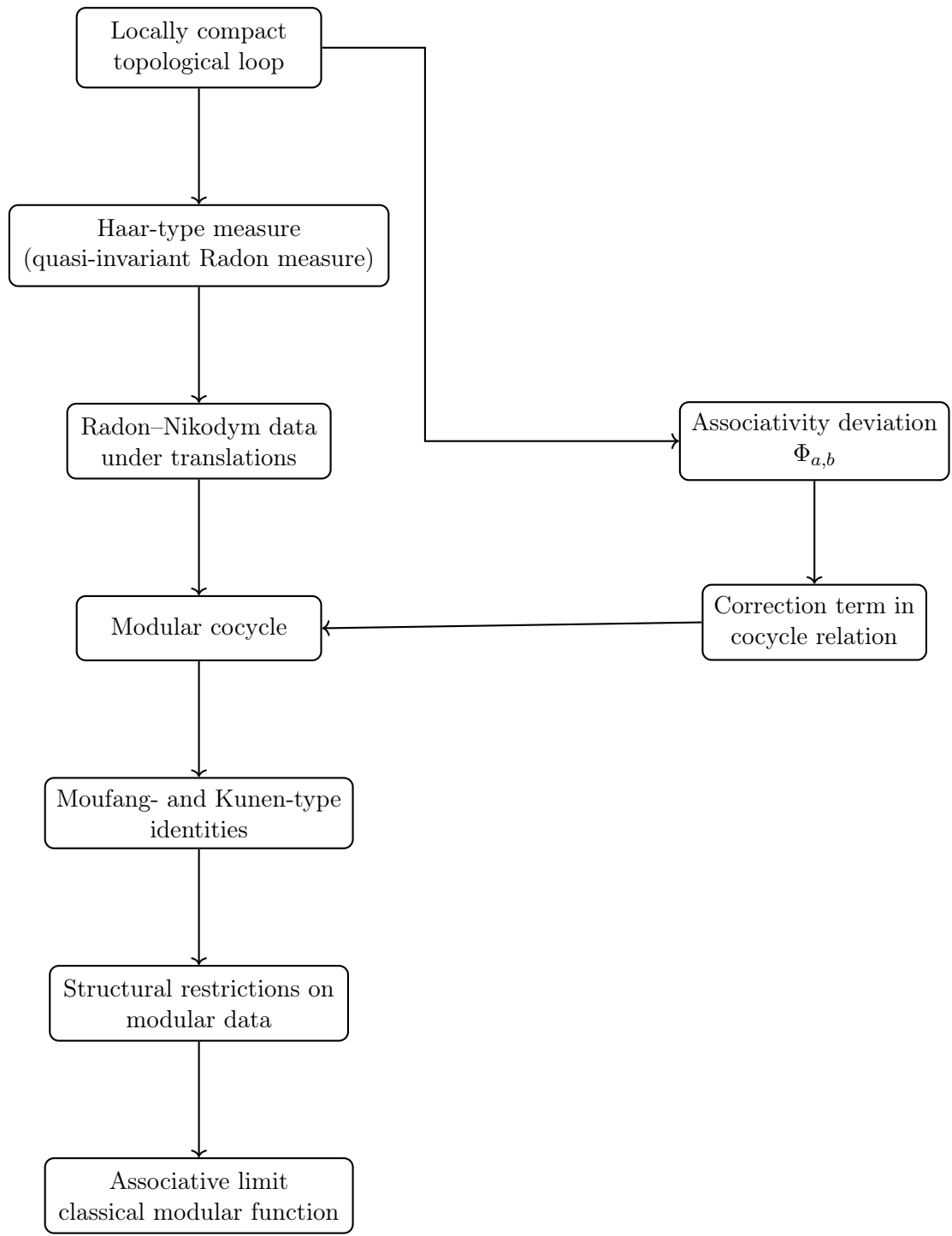

\appendix

\section{Structural Diagram and Relation to the Quasigroup Framework}

The purpose of this appendix is purely expository: it provides a schematic
comparison between the quasigroup theory developed in \cite{InoueQuasi}
and the loop-theoretic modular cocycle framework established in the present paper.

\begin{figure}[H]
\centering
\begin{tikzpicture}[
  scale=0.42,
  transform shape,
  node distance=1.65cm and 2.35cm,
  every node/.style={align=center, font=\LARGE},
  mainbox/.style={
    draw,
    rounded corners,
    thick,
    rectangle,
    minimum width=7.0cm,
    minimum height=2.05cm,
    inner sep=11pt
  },
  smallbox/.style={
    draw,
    rounded corners,
    thick,
    rectangle,
    minimum width=5.8cm,
    minimum height=1.70cm,
    inner sep=9pt
  },
  connector/.style={->, thick},
  dashedconnector/.style={->, thick, dashed}
]

\node[mainbox] (q1) {Topological quasigroups};
\node[mainbox, right=4.0cm of q1] (q2) {Haar-type measures};
\node[mainbox, right=4.0cm of q2] (q3) {Kunen's theorem and\\ quasigroup structure};

\node[smallbox, below=2.1cm of q2] (bridge) {Continuation / extension\\ to the loop setting};

\node[mainbox, below=2.45cm of bridge, xshift=-10.6cm] (l1) {Locally compact\\ topological loops};
\node[mainbox, right=3.6cm of l1] (l2) {Haar-type measure\\ (quasi-invariant Radon measure)};
\node[mainbox, right=3.6cm of l2] (l3) {Radon--Nikodym data\\ and modular cocycle};

\node[smallbox, below=1.9cm of l2] (a1) {Associativity deviation\\ $\Phi_{a,b}$};
\node[smallbox, right=3.6cm of a1] (a2) {Correction term in\\ cocycle relation};

\node[mainbox, below=2.35cm of a1, xshift=-11.6cm] (l4) {Moufang- and Kunen-type\\ identities};
\node[mainbox, right=3.6cm of l4] (l5) {Structural restrictions on\\ modular data};
\node[mainbox, right=3.6cm of l5] (l6) {Associative limit:\\ classical modular function};

\draw[connector] (q1) -- (q2);
\draw[connector] (q2) -- (q3);

\draw[dashedconnector] (q2) -- (bridge);
\draw[connector] (bridge.south west) -- ++(-0.65,-0.95) -- (l1.north);
\draw[connector] (bridge.south) -- (l2.north);
\draw[connector] (bridge.south east) -- ++(0.65,-0.95) -- (l3.north);

\draw[connector] (l1) -- (l2);
\draw[connector] (l2) -- (l3);

\draw[connector] (l2) -- (a1);
\draw[connector] (a1) -- (a2);
\draw[connector] (a2.north) -- ++(0,0.85) -| (l3.south);

\draw[connector] (a1) -- (l4);
\draw[connector] (l4) -- (l5);
\draw[connector] (l5) -- (l6);

\end{tikzpicture}
\caption{Comparison between the quasigroup framework of \cite{InoueQuasi}
and the present loop-theoretic extension. The upper row summarizes the
quasigroup setting, where Haar-type measures were studied in relation to
Kunen's theorem. The lower part summarizes the present paper: Haar-type
measures on locally compact topological loops give rise to Radon--Nikodym
data and modular cocycles, with an additional correction term produced by
associativity deviation. Moufang- and Kunen-type identities then impose
structural restrictions on the modular data, and in the associative limit
the theory recovers the classical modular function.}
\label{fig:quasigroup_loop_overview}
\end{figure}
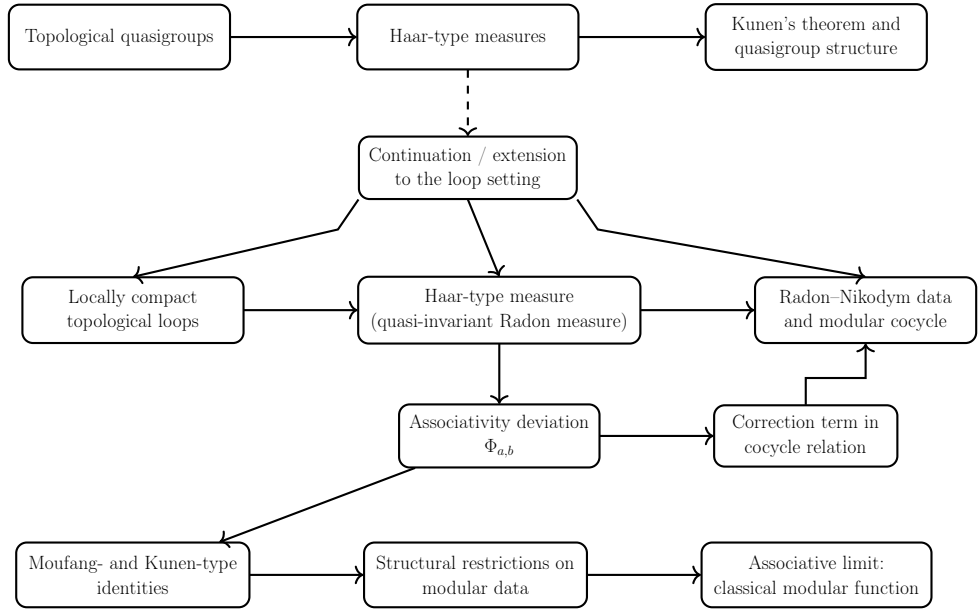

\end{document}